\immediate\write18{biber \jobname}

\documentclass{amsart}

\usepackage{microtype}
\usepackage{amssymb}
\usepackage{mathtools}
\usepackage{tikz-cd}
\usepackage{mathbbol}
\usepackage{csquotes}

\usepackage[
	backend=biber,
	style=alphabetic,
	backref=true,
	url=false,
	doi=false,
	isbn=false,
	eprint=false]{biblatex}

\renewbibmacro{in:}{}  
\AtEveryBibitem{\clearfield{pages}}  

\DeclareFieldFormat{title}{\myhref{\mkbibemph{#1}}}
\DeclareFieldFormat
[article,inbook,incollection,inproceedings,patent,thesis,unpublished]
{title}{\myhref{\mkbibquote{#1\isdot}}}

\newcommand{\doiorurl}{%
	\iffieldundef{url}
	{\iffieldundef{eprint}
		{}
		{http://arxiv.org/abs/\strfield{eprint}}}
	{\strfield{url}}%
}

\newcommand{\myhref}[1]{%
	\ifboolexpr{%
		test {\ifhyperref}
		and
		not test {\iftoggle{bbx:eprint}}
		and
		not test {\iftoggle{bbx:url}}
	}
	{\href{\doiorurl}{#1}}
	{#1}%
}

\usepackage[
	bookmarks=true,
	linktocpage=true,
	bookmarksnumbered=true,
	breaklinks=true,
	pdfstartview=FitH,
	hyperfigures=false,
	plainpages=false,
	naturalnames=true,
	colorlinks=true,
	pagebackref=false,
	pdfpagelabels]{hyperref}

\hypersetup{
	colorlinks,
	citecolor=blue,
	filecolor=blue,
	linkcolor=blue,
	urlcolor=blue
}

\usepackage[capitalize, noabbrev]{cleveref}
\crefname{subsection}{\S\!}{subsections}

\calclayout

\makeatletter
\@namedef{subjclassname@2020}{%
	\textup{2020} Mathematics Subject Classification}
\makeatother

\newtheorem*{theorem*}{Theorem}

\newtheorem*{proposition*}{Proposition}

\newtheorem*{lemma*}{Lemma}

\newtheorem*{corollary*}{Corollary}

\theoremstyle{definition}

\newtheorem*{definition*}{Definition}

\newtheorem*{remark*}{Remark}

\newtheorem*{example*}{Example}

\newtheorem*{construction*}{Construction}

\newtheorem*{convention*}{Convention}

\newtheorem*{terminology*}{Terminology}

\newtheorem*{notation*}{Notation}

\newtheorem*{question*}{Question}

\DeclareMathOperator{\face}{d}

\DeclareMathOperator{\sign}{sign}
\newcommand{\ot}{\otimes}

\newcommand{\N}{\mathbb{N}}
\newcommand{\Z}{\mathbb{Z}}

\newcommand{\R}{\mathbb{R}}

\newcommand{\Sym}{\mathbb{S}}

\newcommand{\Ftwo}{{\mathbb{F}_2}}

\newcommand{\gsimplex}{\mathbb{\Delta}}

\newcommand{\Fun}{\mathsf{Fun}}

\newcommand{\Ch}{\mathsf{Ch}}

\DeclareMathOperator{\Sq}{Sq}

\DeclarePairedDelimiter\bars{\lvert}{\rvert}

\DeclarePairedDelimiter\set{\{}{\}}

\newcommand{\id}{\mathsf{id}}

\newcommand{\op}{\mathrm{op}}

\newcommand{\xla}[1]{\xleftarrow{#1}}

\newcommand{\defeq}{\stackrel{\mathrm{def}}{=}}

\newcommand{\cN}{\mathcal{N}}

\newcommand{\rH}{\mathrm{H}}

\DeclareMathOperator{\assembly}{A}
\DeclareMathOperator{\coMod}{coMod}

\DeclareMathOperator{\copr}{\Delta}
\DeclareMathOperator{\aug}{\epsilon}

\DeclareMathOperator{\chains}{C}
\DeclareMathOperator{\cochains}{C}
\DeclareMathOperator*{\Ot}{\otimes}
\newcommand{\downtriangle}{\text{\Large$\triangledown$}}

\title{Ranicki--Weiss assembly and the Steenrod construction}

\author{Anibal~M.~Medina-Mardones}
\address{A.M-M., Max Planck Institute for Mathematics \and University of Notre Dame}
\email{\href{mailto:ammedmar@mpim-bonn.mpg.de}{ammedmar@mpim-bonn.mpg.de}}

\date{\today}

\dedicatory{In memory of Andrew Ranicki}

\begin{document}
	\newpage

\begin{abstract}
	We show that the Ranicki--Weiss assembly functor, going from chain complex valued presheaves on a simplicial complex to comodules over its Alexander--Whitney coalgebra, factors fully faithfully through the category of comodules over its Steenrod cup-$i$ coalgebra.
\end{abstract}
	\maketitle

\subsection{Introduction}\label{ss:introduction}

We begin with a quote from Andrew's \textit{Algebraic $L$-theory and topological manifolds}.
\begin{displaycquote}{ranicki1992topological}
	Generically, \textit{assembly} is the passage from a local input to a global output.
	The input is usually topologically invariant and the output is homotopy invariant.
	This is the case in the original geometric assembly map of Quinn, and the algebraic $L$-theory assembly map defined here.
\end{displaycquote}
In this note, we will consider the chain complex assembly of Andrew and M. Weiss \cite{ranicki1990assembly}, which can be extended to the $L$-theory assembly functors mentioned in the quote by considering chain complexes with derived Poincar\'e duality.
The connections to assembly via stable homotopy theory \cite{weiss1995asssembly} or equivariant homotopy theory \cite{davis1998assembly} will not be discussed.
In the context of this work, locality is encoded via a simplicial complex $X$, which for simplicity will now be assumed simply-connected.
We will focus on contravariant functors from $X$, regarded as a poset category, to the category of chain complexes $\Ch$.
The assembly of one such functor $\cN$ is its homotopy colimit, defined more precisely as its tensor product $\chains \Ot_X \cN$ with the covariant functor $\chains \colon X \to \Ch$ assigning to a simplex the chains on its closure subcomplex.
Andrew and M. Weiss factored the assembly functor
\[
\assembly \colon \Fun(X^\op, \Ch) \to \coMod_{\chains(X)} \to \Ch
\]
through the category of comodules over the Alexander--Whitney coalgebra of chains on $X$, and constructed, for any finite such comodule $N$, a zig-zag of comodule quasi-isomorphisms
\[
N \to N' \leftarrow \assembly \cN.
\]
For any two functors $\cN,\cN' \colon X^\op \to \Ch$ the chain map
\[
\Fun(X^\op,\Ch)(\cN,\cN') \to \Ch(\assembly \cN, \assembly \cN')
\]
induced by the assembly functor is not surjective in general.
That is to say, the assembly functor is not full.
Indeed, the morphisms in the domain are defined locally, whereas in the target they are global in nature.
The purpose of this article is to describe a subcategory $\coMod^{\Sym_2}_{\chains(X)}$ of $\coMod_{\chains(X)}$ and a further factorization
\[
\assembly \colon \Fun(X^\op, \Ch) \to
\coMod^{\Sym_2}_{\chains(X)} \to
\coMod_{\chains(X)} \to
\Ch
\]
with the first arrow a fully faithful functor.

Let us now motivate and describe the category $\coMod^{\Sym_2}_{\chains(X)}$.
For Andrew's $L$-theory assembly functors, it is important to provide the objects in $\Fun(X^\op, \Ch)$ with derived Poincar\'e duality.
For example, let us consider a (geometric) Poincar\'e duality complex $X$ with a preferred fundamental cycle, also denoted by $X$.
The Alexander--Whitney diagonal $\copr_0 \colon \chains(X) \to \chains(X) \ot \chains(X)$ defines a Poincar\'e duality quasi-isomorphism
\[
\begin{tikzcd}[column sep=tiny,row sep=0]
	\cochains^{-k}(X) \arrow[r] & \chains_{n-k}(X) \\
	\alpha \arrow[r,mapsto] & (\alpha \ot \id)\copr_0(X).
\end{tikzcd}
\]
Using $T\copr_0$ instead of $\copr_0$ above, where $T$ is the transposition of tensor factors, one obtains a new Poincar\'e duality quasi-isomorphism.
These two quasi-isomorphisms are different in general since $\copr_0$ and $T\copr_0$ are not equal, but they are chain homotopic.
Steenrod \cite{steenrod1947products} fully derived the $\Sym_2$ symmetry of the diagonal explicitly constructing natural ``higher diagonals'' $\copr_i$ with $\partial \copr_{i+1} = (1 \pm T) \copr_i$ and $\copr_{-1} = 0$.
The evaluation of the fundamental cycle on this richer structure provides $\chains(X)$ with the structure of a (symmetric) algebraic Poincar\'e complex as defined by Andrew.

We will focus on the Steenrod cup-$i$ coalgebra structure on $\chains(X)$ for $X$ not necessarily a Poincar\'e duality complex.
This structure can be expressed as an $\Sym_2$-equivariant chain map
\[
\copr \colon W \ot \chains(X) \to \chains(X) \ot \chains(X)
\]
where $W$ is the minimal free resolution of $\Z$ by $\Z[\Sym_2]$-modules.
We will show that the assembly of any functor $\cN \colon X^\op \to \Ch$ is equipped with a chain map
\[
\nabla \colon W \ot \assembly \cN \to \chains(X) \ot \assembly \cN
\]
specializing to the comodule structure of Andrew and M. Weiss.
The pair $(\assembly \cN, \nabla)$ is an object in the category $\coMod^{\Sym_2}_{\chains(X)}$ whose morphisms are given by linear maps $f \colon N \to N'$ making the diagram
\[
\begin{tikzcd}
	W \ot N \arrow[d, "\triangledown"] \arrow[r, "\id\, \ot f"] &
	W \ot N' \arrow[d, "\triangledown'"] & \\
	\chains(X) \ot N \arrow[r, "f \ot\, \id"] &
	\chains(X) \ot N'
\end{tikzcd}
\]
commute.
We can now state the contribution of this paper.

\begin{theorem*}
	Let $X$ be a simply-connected simplicial complex.
	The Ranicki--Weiss assembly functor to comodules over the Alexander--Whitney coalgebra of $X$ factors through comodules over its Steenrod cup-$i$ coalgebra
	\[
	\Fun(X^\op, \Ch) \to \coMod^{\Sym_2}_{\chains(X)} \to \coMod_{\chains(X)}
	\]
	with the first arrow a fully faithful functor and the second a forgetful one.
\end{theorem*}

This statement is 1-categorical, the $(\infty,1)$-categorical version will be discussed elsewhere.
In the $(\infty,1)$-categorical setting we mention the paper \cite{rivera2020system} where a connection between assembly and Brown's twisted tensor product is explored.
The goal of the our theorem above is to provide an algebraic model of $\Fun(X^\op,\Ch)$ that retains its locality properties, serving as a possible step in the path connecting algebraic models of homotopy types \cite{quillen1969rational, sullivan1977infinitesimal,mandell2001padic} and Andrew's total surgery obstruction for topological manifold structures \cite{ranicki1979obstruction,ranicki1992topological,macko2013obstruction}.

\subsection*{Acknowledgement}

I would like to remember Andrew's kind and joyful personality which I enjoyed during many pleasant conversations as a graduate student.
I am also indebted to Dennis Sullivan who introduced me to Andrew's work and suggested the ``Ranicki gambit'' as my thesis project, the technical core of which is the result presented here.
Andrew's mathematics and character had a formative influence on me and on many others.
He will be missed.

\subsection{Conventions}

We write $(\Ch, \ot, \Z)$ for the symmetric monoidal category of chain complexes of free $\Z$-modules, $\Sym_2$ for the group with only one non-identity element $T$, and $X$ for an unspecified simplicial complex.
In this work, the vertices of a simplicial complex are equipped with a partial order compatibly restricting to a total order on each simplex.
We denote the usual integral chains of $X$ by $\chains(X)$.

\subsection{Presheaves}\label{ss:presheaves}

Regard $X$ as a poset category with a morphism $x \to y$ whenever $x$ is a subsimplex of $y$.
The category of \textit{presheaves on} $X$ is the category of functors $\Fun(X^\op, \Ch)$.
Given a presheaf $\cN$ on $X$, we will denote the chain complex it associates to a simplex $x$ by $\cN_x$ and the linear map associated to a morphism $x \to y$ by $\cN_{x \to y}$.
We remark that this category is enriched over $\Ch$.
These presheaves were termed ``local systems'' in \cite{ranicki1990assembly} where $X$ was allowed to be a semi-simplicial set.
We do not use this terminology since the associated topological sheaf --which we describe next-- does not need to be locally constant.

\subsection{Remark on topological sheaves}\label{ss:topological sheaf}

The set of simplices of $X$ can be endowed with a topology making presheaves as defined above into topological sheaves.
Explicitly, this topology on the set of simplices of $X$ is generated by subsets $\bar y = \set{x \in X \mid x \to y}$.
Notice that as poset categories $X$ and the category of open sets of this topology are isomorphic.
Therefore, so are their categories of chain complex valued presheaves.
It can be verified that these topological presheaves are indeed sheaves.

\subsection{Symmetric coalgebras}

Let $W$ be the minimal free resolution of $\Z$ by free $\Z[\Sym_2]$-modules:
\[
\Z[\Sym_2]\set{e_0} \xla{1-T} \Z[\Sym_2]\set{e_1} \xla{1+T} \Z[\Sym_2]\set{e_2} \xla{1-T} \dotsb.
\]
The data of a \textit{symmetric coalgebra} is an augmented chain complex $\varepsilon \colon C \to \Z$ and an $\Sym_2$-equivariant chain map
\[
\triangle \colon W \ot C \to C \ot C,
\]
where the action of $\Sym_2$ in the domain is induced from that in $W$ and in the target is given by transposition.
We will write $\triangle_i$ for $\triangle(e_i \ot -)$ and demand the triple $(C, \triangle_0, \varepsilon)$ to be a coalgebra, i.e.,
a comonoid in $\Ch$.
Notice that
\[
\partial(\triangle_i) = \big( 1+(-1)^i T \big) \triangle_{i-1}
\]
with the convention $\triangle_{-1} = 0$.

\subsection{Remark on higher categories}
\label{ss:higher categories}

An $\omega$-category is the limiting concept obtained from the recursive definition of an $n$-category as a category enriched in $(n-1)$-categories.
Building on ideas from \cite{brown1981cubes, kapranov1991polycategory, steiner2004omega}, in \cite{medina2020globular} a functor from a full subcategory of symmetric coalgebras to $\omega$-categories was constructed and shown to be an equivalence onto a full subcategory recently characterized intrinsically in \cite{ozornova2022steiner}.
Street's orientals \cite{street1987orientals} --which define the nerve of $\omega$-categories-- are obtained by applying this functor to the chains on standard simplices with Steenrod's cup-$i$ coalgebra structure reviewed in \cref{ss:cup-i}.

\subsection{Diagonal, augmentation and join}

The \textit{Alexander--Whitney coproduct}
\[
\copr_0 \colon \chains(X) \to \chains(X) \ot \chains(X)
\]
is defined on a basis element $[v_0, \dots, v_n]$ by
\begin{equation*}\label{e:alexander-whitney coalgebra}
	\copr_0 \big( [v_0, \dots, v_n] \big) =
	\sum_{k=0}^n \ [v_0, \dots, v_k] \ot [v_k, \dots, v_n].
\end{equation*}
We mention that the triple $(\chains(X), \copr_0, \aug)$ where $\aug \colon \chains(X) \to \Z$ is the usual \textit{augmentation map} is a coalgebra.


The \textit{join product}
\[
\ast \colon \chains(X) \ot \chains(X) \to \chains(X)
\]
is the natural degree~$1$ linear map defined on a basis element
\[
[v_0, \dots, v_p] \ot [v_{p+1}, \dots, v_q]
\]
to be $0$ if $\set{v_\ell : \ell = 0, \dots, q}$ is not the set of vertices of a simplex in $X$ or if $v_i = v_j$ for some $i \neq j$,
otherwise
\[
\ast \big(\left[v_0, \dots, v_p \right] \ot \left[v_{p+1}, \dots, v_q\right]\big) =
(-1)^{p} \sign(\pi) \left[v_{\pi(0)}, \dots, v_{\pi(q)}\right]
\]
where $\pi$ is the permutation that orders the vertices.


\subsection{Steenrod construction}\label{ss:cup-i}

We will now describe Steenrod's construction of cup-$i$ coproducts \cite[p.293]{steenrod1947products} making $\chains(X)$ into a symmetric coalgebra.
Let
\[
\copr \colon W \ot \chains(X) \to \chains(X) \ot \chains(X)
\]
be the $\Sym_2$-equivariant chain map determined by the linear maps $\copr_i = \copr(e_i \ot -)$, referred to as \textit{cup-$i$ coproducts}, recursively defined for $i > 1$ by
\begin{equation*}\label{e:cup-i coproducts}
	\copr_i =
	(\ast \ot \id) \circ (\id \ot T\copr_{i-1}) \circ \copr_0.
\end{equation*}
The triple $(\chains(X), \copr, \aug)$ forms a symmetric coalgebra which we refer to as the \textit{Steenrod cup-$i$ coalgebra} of $X$.

Alternative descriptions of cup-$i$ coproducts, all shown to be equivalent to Steenrod's original via an axiomatic characterization \cite{medina2022axiomatic}, can be found in \cite{real1996computability, gonzalez-diaz1999steenrod, mcclure2003multivariable, medina2021fast_sq}.
These axiomatized cup-$i$ coproducts seem to be combinatorially fundamental, being constructed from the convex geometry of the standard simplex $\gsimplex^n$ in $\R^n$ using a generic orthonormal basis \cite{medina2022fib_poly}.

\subsection{Special cases}\label{ss:special_cases}

We will use the following special values of the Steenrod cup-$i$ coalgebra.
For a basis element $x$ of $\chains(X)_n$ and $i > n$ we have:
\begin{align}
	&\copr_i(x) = 0, \\
	&\copr_n(x) = \pm \, x \ot x, \\
	&\copr_{n-1}(x) =
	\sum_{u \text{ even}} \pm \, x \ot \face_u(x) \ +
	\sum_{u \text{ odd}} \pm \face_u(x) \ot x.
\end{align}
where $\pm$ stands for an unspecified sign and $\face_u [v_0,\dots,v_n] = [v_0,\dots,\widehat{v}_u,\dots,v_n]$.

\subsection{Remark on Steenrod squares}\label{ss:steenrod squares}

We mention that Steenrod's introduction of the cup-$i$ coproducts $\copr_i$ was motivated by the desire to construct finer invariants on the cohomology of spaces with mod 2 coefficients.
These are the celebrated \textit{Steenrod squares}
\[
\Sq^k \colon \rH^\bullet(X; \Ftwo) \to \rH^\bullet(X; \Ftwo)
\]
defined on a class $[\alpha]$ of cohomological degree $n$ by
\[
\Sq^k [\alpha] = \big[ (\alpha \ot \alpha) \copr_{n-k}(-) \big].
\]

\subsection{Remark on $E_\infty$-structures}

Although we do not use the following fact in this work, we mention that the cup-$i$ coproducts of Steenrod are part of much larger structure derived from the symmetry of the diagonal of spaces; that of an $E_\infty$-coalgebra structure.
Explicitly, this structure is given by all maps $\chains(X) \to \chains(X)^{\ot r}$ for any $r > 0$ obtained from compositions of the Alexander--Whitney coproduct and the join product.
The underlying model for the $E_\infty$-operad use in this statement is presented in \cite{medina2020prop1, medina2021prop2}.
We mention that this $E_\infty$-coalgebra structure on simplicial chains is compatible with the $E_\infty$-coalgebras of McClure--Smith \cite{mcclure2003multivariable} and Berger--Fresse \cite{berger2004combinatorial}.
We also mention that analogues of the cup-$i$ coproducts modeling Steenrod operations at odd primes were defined in \cite{medina2021may_st} and implemented in \cite{medina2021comch}.
For a discussion using cubes instead of simplices please consult \cite{medina2021cubical}.

\subsection{Comodules}

The data of a \textit{comodule} over a symmetric coalgebra $(C, \triangle, \aug)$ is a chain complex $N$ together with a chain map
\[
\downtriangle \colon W \ot N \to N \ot C.
\]
We will write $\downtriangle_i$ for $\downtriangle(e_i \ot -)$ and $\downtriangle_i^T$ for $\downtriangle(Te_i \ot -)$, and demand that $(N, \downtriangle_{\!0})$ be a comodule over the coalgebra $(C, \triangle_0, \aug)$.

We will denote by $\coMod_C^{\Sym_2}$ the category of such comodules with strict morphisms, i.e. linear maps $f \colon N \to N'$ making the following diagram
\[
\begin{tikzcd}
	W \ot N \arrow[d, "\triangledown"] \arrow[r, "\id\, \ot f"] &
	W \ot N' \arrow[d, "\triangledown'"] & \\
	N \ot C \arrow[r, "f \ot\, \id"] &
	N' \ot C
\end{tikzcd}
\]
commute.
We remark that this category is enriched in $\Ch$.

Although we do not use this fact, we mention that the examples of comodules we will construct make the following diagram commute up to homotopy
\[
\begin{tikzcd}
	W \ot W \ot N \arrow[r, "\id\, \ot \triangledown"] \arrow[d, "\id\, \ot \triangledown"'] &
	W \ot N \ot C \arrow[d, "(\id\, \ot \triangle)(T \ot \id)"] \\
	W \ot N \ot C \arrow[r, "\triangledown \ot\, \id"] &
	N \ot C \ot C.
\end{tikzcd}
\]

\subsection{Closure cosheaf}

For any simplex $x$ in $X$ denote by $\bar x$ the subcomplex of $X$ containing the subsimplices of $x$.
The \textit{closure cosheaf} is the (covariant) functor from the poset category of $X$ to $\Ch$ defined on objects by $x \mapsto \chains(\bar x)$ and on morphisms by canonical inclusions $\chains_{x \to y} \colon \chains(\bar x) \to \chains(\bar y)$.

For any simplex $y$ in $X$ and subsimplex $x$, we will find it convenient to identify the basis element $x$ in $\chains(\bar y)$ with the unique morphism $(x \to y)$.
In particular, we have $\chains_{y \to z} (x \to y) = (x \to z)$.

\subsection{Assembly}\label{ss:assembly}

The \textit{assembly functor}
\[
\assembly \colon \Fun(X^\op, \Ch) \to \Ch
\]
is given by tensoring on $X$ with the closure cosheaf.
Explicitly, let $\cN$ be a sheaf on $X$, then
\[
\assembly \cN = \bigoplus_{x \in X} \chains(\bar x) \ot \cN_x \ / \sim
\]
where for simplices $x \to y \to z$ and $c \in \cN[z]$ we have
\[
(x \to z) \ot c \ \sim \, (x \to y) \ot \cN_{y \to z}(c)
\]
and for a morphism $F \colon \cN \to \cN'$ of presheaves on $X$ we have
\[
\assembly F[(x \to y) \ot c] = [(x \to y) \ot F_y(c)].
\]
We remark that this is a functor of categories enriched in $\Ch$.

If $X$ is simply-connected, this definition agrees with the one given by Andrew and M. Weiss in \cite[Definition 1.4]{ranicki1990assembly}.
We will use the definition above for $X$ not necessarily simply-connected, remarking that it can be applied to the the universal cover of $X$ to recover their definition.

\subsection{Lift to comodules}

We will describe a factorization of the assembly functor
\[
\assembly \colon \Fun(X^\op, \Ch) \to \coMod^{\Sym_2}_{\chains(X)} \to \Ch
\]
where the second arrow is the forgetful functor.

Let $\cN$ be a presheaf on $X$ and let
\[
\nabla \colon W \ot \assembly \cN \to \chains(X) \ot \assembly \cN
\]
be the chain map defined by the linear maps
\[
\nabla_i \defeq \nabla(e_i \ot -)
\quad \text{ and } \quad
\nabla_i^T \defeq \nabla(Te_i \ot -)
\]
which on a pair of simplices $x \to y$ and an element $c \in \cN_y$ are given by
\[
\nabla_i \big( [(x \to y) \ot c] \big) =
\sum_{\lambda \in \Lambda} \alpha_\lambda \, x_\lambda^{(1)} \ot [(x_\lambda^{(2)} \to y) \ot c]
\]
and
\[
\nabla_i^T \big( [(x \to y) \ot c] \big) =
\sum_{\lambda \in \Lambda} \alpha^T_\lambda \, x_\lambda^{(2)} \ot [(x_\lambda^{(1)} \to y) \ot c],
\]
where
\[
\copr_i(x) = \sum_{\lambda \in \Lambda} \alpha_\lambda \, x_\lambda^{(1)} \ot x_\lambda^{(2)}
\]
and
\[
T \copr_i(x) =
\sum_{\lambda \in \Lambda} \alpha^T_\lambda \, x_\lambda^{(2)} \ot x_\lambda^{(1)}.
\]

\begin{lemma*}
	The map $\nabla$ naturally makes the assembly of a presheaf on $X$ into a comodule over the symmetric coalgebra $\chains(X)$.
\end{lemma*}

\begin{proof}
	The only part of the statement that is not immediate is the claim that $\nabla$ is well defined.
	For $i \in \N$, consider a presheaf $\cN$ on $X$, simplices $x \to y \to z$ and an element $c \in \cN_z$, then
	\begin{align*}
		\nabla_i [(x \to z) \ot c] &=
		\sum \alpha_\lambda \, x^{(1)}_\lambda \ot {[(x^{(2)}_\lambda \to z) \ot c]} \\ &=
		\sum \alpha_\lambda \, x^{(1)}_\lambda \ot [(x^{(2)}_\lambda \to y) \ot \cN_{y \to z}(c)] \\ &=
		\nabla_i [(x \to y) \ot \cN_{y \to z}(c)].
	\end{align*}
	A similar argument applies to $\nabla^T_i$ and completes the proof.
\end{proof}

A similar construction was presented by Andrew and M. Weiss in \cite[Proposition~5.3]{ranicki1990assembly}.
In the simply-connected case, $\nabla_0$ agrees with their coaction of the Alexander--Whitney coalgebra of $X$ on $\assembly \cN$.

\subsection{Full embedding}\label{ss:full embedding}

We now prove the main result of this work, for which we are \textbf{not} assuming the simplicial complex $X$ to be simply-connected.

\begin{theorem*}
	The assembly functor from chain complex valued presheaves on $X$ to comodules over the symmetric coalgebra $\chains(X)$ is full and faithful.
\end{theorem*}

\begin{proof}
	Consider two presheaves $\cN$ and $\cN'$ over $X$ and the function
	\[
	\Fun(X^\op, \Ch)(\cN, \cN') \to \coMod^{\Sym_2}_{\chains(X)}(\assembly \cN, \assembly \cN')
	\]
	induced by the assembly functor, which we will show to be a bijection.

	For a simplex $x$ in $X$ we write throughout this proof $x$ for the basis element it represents in both $\chains(X)$ and $\chains(\bar x)$, where it corresponds to the identity $(x \to x)$.

	We start by showing that the chain map above is injective.
	Consider a morphism of presheaves $F \colon \cN \to \cN'$ over $X$ with $\assembly F = 0$.
	For each simplex $x$ choose a basis $B_x$ of $\cN_x$ and notice that
	\[
	B = \set[\big]{[x \ot b] \mid x \in X, b \in B_x}
	\]
	is a basis for $\assembly \cN$.
	By assumption we have $\assembly F [x \ot b] \defeq [x \ot F_x(b)] = 0$ for each $[x \ot b] \in B$, which implies $F_x(b) = 0$ and consequently $F = 0$.

	Let us now prove the surjectivity of this map.
	Consider a morphism of comodules $f \colon \assembly \cN \to \assembly \cN'$.
	We will construct a morphism of presheaves $F \colon \cN \to \cN'$ such that $\assembly F = f$.
	Consider $x \in X$, $c \in \cN_x$ and the class $[x \ot c] \in \assembly \cN$.
	Its image under $f$ is of the form
	\[
	f \big( [x \ot c] \big) = \sum_{\lambda \in \Lambda} \, [x_\lambda \ot c_\lambda]
	\]
	where $x_\lambda \in X$ and $c_\lambda \in \cN'_{x_\lambda}$.
	We will first show that for each $\lambda$ in the (finite) sum above the dimension of the simplex $x_\lambda$ satisfies $\bars{x_\lambda} \leq \bars{x}$.
	Let $i = \max \set[\big]{\bars{x_\lambda} : \lambda \in \Lambda}$ and $\Lambda_i = \set[\big]{\lambda \in \Lambda : \bars{x_\lambda} = i}$.
	Using \cref{ss:special_cases} we have $\copr_i(x_{\lambda}) = \pm \, x_{\lambda} \ot x_{\lambda}$ for $\lambda \in \Lambda_i$ and $\copr_i(x_{\lambda}) = 0$ for $\lambda \notin \Lambda_i$.
	Assume $i > \bars{x}$ so $\copr_i(x) = 0$.
	Therefore, using that $f$ is a comodule map we have
	\begin{align*}
		0 &=
		(\id \ot f) \circ \nabla_i \big( [x \ot c] \big) \\&=
		\nabla_i \circ f \big( [x \ot c] \big) \\&=
		\sum_{\lambda \in \Lambda} \nabla_i \big( [x_\lambda \ot c_\lambda] \big) \\&=
		\sum_{\lambda \in \Lambda_i} \pm \, x_\lambda \ot [x_\lambda \ot c_\lambda],
	\end{align*}
	which implies $[x_\lambda \ot c_\lambda] = 0$ and consequently $c_{\lambda} = 0$ for each $\lambda \in \Lambda_i$.

	We will now show that $\Lambda$ has a single element $\lambda$ with $x_\lambda = x$.
	Let $i = \bars{x}$, so
	\begin{align*}
		(\id \ot f) \circ \nabla_i \big( [x \ot c] \big) &=
		\pm \, x \ot f \big( [x \ot c] \big) \\&=
		\sum_{\lambda \in \Lambda} \pm \, x \ot [x_\lambda \ot c_\lambda]
	\end{align*}
	and
	\[
	\nabla_i \circ f \big( [x \ot c] \big) =
	\sum_{\lambda \in \Lambda_i} \pm \, x_\lambda \ot [x_\lambda \ot c_\lambda].
	\]
	Using that $f$ is a comodule map we have
	\[
	\sum_{\lambda \in \Lambda} \pm \, x \ot [x_\lambda \ot c_\lambda] \, =
	\sum_{\lambda \in \Lambda_i} \pm \, x_\lambda \ot [x_\lambda \ot c_\lambda],
	\]
	from which the claim follows.
	Using this we can define a collection of maps $F = \set[\big]{F_x \colon \cN_x \to \cN'_x}$ parameterized by the simplices of $X$ by the identity $[x \ot F_x(c)] = f \big( [x \ot c] \big)$.
	We claim that $F$ is a well defined morphism of presheaves on $X$.
	That is to say, that for any morphism $x \to y$ and $c \in \cN_y$ we have
	\[
	\cN'_{x \to y} \circ F_y(c) = F_x \circ \cN_{x \to y}(c).
	\]
	It is clear that if this is the case then $\assembly F = f$.
	To prove this claim, it suffices to consider morphisms $\face_u(x) \to x$ where $\face_u(x)$ is obtained by removing the vertex of $x$ in position $u$.
	Let $i = \bars{x}-1$ and recall from \cref{ss:special_cases} that
	\[
	\copr_i(x) \ =
	\sum_{u \text{ even}} \pm \, x \ot \face_u(x) \ +
	\sum_{u \text{ odd}} \pm \face_u(x) \ot x.
	\]
	Therefore, on one hand we have that $(\id \ot f) \circ \nabla_i \big( [x \ot c] \big)$ is equal to
	\begin{align*}
		&\sum_{u \text{ even}} \pm \, x \ot [\face_u(x) \ot F_{\face_u(x)} \circ \cN_{\face_u(x) \to x}(c)] \ +
		\sum_{u \text{ odd}} \pm \face_u(x) \ot [x \ot F_x(c)],
	\end{align*}
	while on the other we have that $\nabla_i \circ f \big( [x \ot c] \big) \defeq \nabla_i [x \ot F_x(c)]$ is equal to
	\begin{align*}
		&\sum_{u \text{ even}} \pm \, x \ot [\face_u(x) \ot \cN'_{\face_u(x) \to x} \circ F_x(c)] \ +
		\sum_{u \text{ odd}} \pm \face_u(x) \ot [x \ot F_x(c)].
	\end{align*}
	Using the fact that $f$ is a comodule map we can deduce that
	\[
	\sum_{u \text{ even}} \pm [\face_u(x) \ot F_{\face_u(x)} \circ \cN_{\face_u(x) \to x}(c)] \ =
	\sum_{u \text{ even}} \pm [\face_u(x) \ot \cN'_{\face_u(x) \to x} \circ F_x(c)]
	\]
	and, consequently, that for $u$ even
	\[
	F_{\face_u(x)} \circ \cN_{\face_u(x) \to x}(c) =
	\cN'_{\face_u(x) \to x} \circ F_{x}(c)
	\]
	as desired.
	For $u$ odd we repeat the same argument using $\nabla_i^T$ instead of $\nabla_i$.
	This concludes the proof of our main theorem.
\end{proof}
	\sloppy
	\printbibliography
\end{document}